\documentclass[11pt]{article}

\usepackage[latin1]{inputenc}
\usepackage[english]{babel}
\usepackage{here,enumerate,graphicx}
\usepackage{amsmath,amssymb,amscd,amsthm,mathrsfs}

\usepackage[flushmargin]{footmisc}
\newcommand\blfootnote[1]{%
	\begingroup
	\renewcommand\thefootnote{}\footnote{#1}%
	\addtocounter{footnote}{-1}%
	\endgroup}

\theoremstyle{plain}
\newtheorem{theorem}{Theorem}[section]
\newtheorem{proposition}[theorem]{Proposition}

\newtheorem{lemma}[theorem]{Lemma}
\newtheorem{corollary}[theorem]{Corollary}

\theoremstyle{definition}
\newtheorem{definition}[theorem]{Definition}
\newtheorem{remark}[theorem]{Remark}

\newtheorem{question}[theorem]{Question}

\newcommand{\NN}{\mathbb{N}}
\newcommand{\RR}{\mathbb{R}}
\newcommand{\CC}{\mathbb{C}}

\newcommand{\Ac}{\mathcal{A}}
\newcommand{\AP}{\mathcal{AP}}
\newcommand{\Bc}{\mathcal{B}}
\newcommand{\Cc}{\mathcal{C}}
\newcommand{\Fc}{\mathcal{F}}
\newcommand{\Kc}{\mathcal{K}}
\newcommand{\Nc}{\mathcal{N}}
\newcommand{\Part}{\mathscr{P}}
\newcommand{\Tc}{\mathcal{T}}

\newcommand{\Vc}{\mathcal{V}}
\newcommand{\Yc}{\mathcal{Y}}
\newcommand{\Zc}{\mathcal{Z}}

\newcommand{\Rec}{\text{Rec}}

\newcommand{\cl}{\overline}
\newcommand{\eps}{\varepsilon}

\newcommand{\com}{\overline}
\newcommand{\fuz}{\hat}
\newcommand{\con}{\tilde}

\begin{document}
\begin{center}
	\begin{LARGE}
		{\bf Recurrence in collective dynamics: From the\\ hyperspace to fuzzy dynamical systems}
	\end{LARGE}
\end{center}

\begin{center}
	\begin{Large}
		Illych Alvarez, Antoni L\'opez-Mart\'inez \& Alfred Peris\blfootnote{\textbf{Journal-ref}: Fuzzy Sets and Systems, Volume 506, article number 109296, (2025).\\ \textbf{DOI}: https://doi.org/10.1016/j.fss.2025.109296}
	\end{Large}
\end{center}


\begin{abstract}
	We study for a dynamical system $f:X\longrightarrow X$ some of the principal topological recurrence-kind properties with respect to the induced maps $\com{f}:\Kc(X)\longrightarrow\Kc(X)$, on the hyperspace of non-empty compact subsets of $X$, and $\fuz{f}:\Fc(X)\longrightarrow\Fc(X)$, on the space of normal fuzzy sets consisting of the upper-semicontinuous functions $u:X\longrightarrow [0,1]$ with compact support and such that $u^{-1}(\{1\})\neq\varnothing$. In particular, we characterize the properties of topological and multiple recurrence for the extended systems $(\Kc(X),\com{f})$ and $(\Fc(X),\fuz{f})$, which cover the cases of the so-called nonwandering and Van der Waerden systems. Special attention is given to the case where the underlying space is completely metrizable, for which we obtain some stronger point-recurrence equivalences.
\end{abstract}


\section{Introduction}

A {\em dynamical system} is a pair $(X,f)$ formed by continuous map $f:X\longrightarrow X$ acting on a topological space~$X$, which is usually called the {\em phase space}. In this paper we focus on the interplay between {\em individual dynamics} (i.e.\ the evolution of the system acting on points of the phase space) and the so-called {\em collective dynamics} (i.e.\ the evolution of the system acting on subsets of the phase space).

This kind of question goes back to the classical work of Bauer and Sigmund~\cite{BaSig1975} where, in the context of compact metric spaces, they studied the connection (for properties such as {\em distality}, {\em transitivity}, {\em mixing} and {\em weak-mixing}) between $(X,f)$ and the induced system $(\Kc(X),\com{f})$, on the hyperspace $\Kc(X)$ of non-empty compact subsets of $X$ endowed with the Vietoris topology. From then on, the dynamics on the hyperspace of compact sets (even over not necessarily metric spaces) have captured the attention of many researchers (see for instance \cite{Banks2005,BerPeRo2017,LiWaZha2006,Peris2005} and the literature cited therein).
 
In addition, when the underlying space $X$ is metrizable, another type of collective dynamics has recently been considered: the dynamical system $(X,f)$ induces a system $(\Fc(X),\fuz{f})$, on the space $\Fc(X)$ of {\em normal fuzzy sets}, where $\fuz{f}:\Fc(X)\longrightarrow\Fc(X)$ is called the {\em Zadeh extension} of~$f$.~In this context: Jard\'on et al.\ studied in~\cite{JarSan2021_FS,JarSanSan2020a,JarSanSan2020b} the relations between $(X,f)$, $(\Kc(X),\com{f})$ and $(\Fc(X),\fuz{f})$~with respect to some {\em expansive} and {\em transitivity} properties; C\'anovas and Kupka in~\cite{CaKup2017}, and Kim et al.\ in~\cite{KimChenJu2017,RiKimJu2023} focused on the notion of {\em topological entropy}; and other dynamical properties, including the so-called Devaney, Li-Yorke and distributional chaos, were studied in~\cite{Kup2011,MarPeRo2021,WangWei2012,WuDingLuWang2017}.

The objective of this paper is developing the respective theory of collective dynamics but for the principal topological recurrence-kind properties, which, surprisingly, have not been previously explored even within the traditional framework of the hyperspace extension $(\Kc(X),\com{f})$ on the space of non-empty compact subsets of $X$. This lack of theory is very striking because {\em recurrence} is one of the fundamental, oldest and most investigated concepts in the area of dynamical systems. Indeed, we can date the beginning of this notion at the end of the 19th century when Poincar\'e introduced the later called {\em Poincar\'e recurrence theorem} in the context of Ergodic Theory.

If we focus on Topological Dynamics, as we do along this paper, the appearance of recurrence goes back to the works of Gottschalk and Hedlund in 1955~\cite{GottHed1955_book}, and of Furstenberg in 1981~\cite{Furstenberg1981_book}, together with some more recent advances by authors such as Banks~\cite{Banks1999}, Glasner~\cite{Glasner2004} and Kwietniak et al.~\cite{KwiLiOpYe2017}.

The paper is organized as follows. In Section~\ref{Sec_2:notation} we introduce the notation together with the general background needed. We recall there both the dynamical properties considered, including nonwandering and Van der Waerden systems (see Remark~\ref{Rem:nonwandering.VanDerWaerden} below), and the specific topological setting that we use for the hyperspaces $\Kc(X)$ and $\Fc(X)$. In Section~\ref{Sec_3:K(X)} we focus on the system $(\Kc(X),\com{f})$ where~$X$ is a topological space, in Section~\ref{Sec_4:F(X)} we assume that $X$ is a metric space to compare the recurrence of~$(\Kc(X),\com{f})$ with that of $(\Fc(X),\fuz{f})$, and in Section~\ref{Sec_5:complete} we let $X$ be a complete metric space to obtain stronger dynamical results in terms of various point-recurrence properties. In Section~\ref{Sec_6:conclusions} we include some open problems asking whether some of the complete spaces results still hold in the non-complete case, also if our main results (Theorems~\ref{The:K.rec}~and~\ref{The:F.rec}) can be extended to linear dynamical systems, and we finally discuss the possibility of considering other topologies on $\Kc(X)$ and $\Fc(X)$ as in \cite{WangWei2012}.

\section{Notation and general background}\label{Sec_2:notation}

In this paper we combine the theory of dynamical systems with that of hyperspaces of compact and fuzzy sets, and we recall in this section the basic definitions and notation that we use through the whole paper. We start by the dynamical properties that we are about to study, and then we include some basic facts about the topologies that we consider on the hyperspaces $\Kc(X)$ and $\Fc(X)$. From now on we let $\NN$ be the set of strictly positive integers and we will write $\NN_0 := \NN \cup \{0\}$.

\subsection{Dynamical properties: From transitivity to recurrence}

As stated in the Introduction, along this paper a {\em dynamical system} will be a pair $(X,f)$ formed by a continuous map~$f:X\longrightarrow X$ acting on a topological space $X$. Given any positive integer $N \in \NN$ we will denote by $f_{(N)}:X^N\longrightarrow X^N$ the {\em $N$-fold direct product} of $f$ with itself, i.e.\ the pair $(X^N,f_{(N)})$ will be the dynamical system
\[
f_{(N)} := \underbrace{f\times\cdots\times f}_{N} : \underbrace{X\times\cdots\times X}_{N} \longrightarrow \underbrace{X\times\cdots\times X}_{N},
\]
where $X^N:=X\times\cdots\times X$ is the {\em $N$-fold direct product} of $X$ and $f_{(N)}\left((x_1,...,x_N)\right) := (f(x_1),...,f(x_N))$ for each $N$-tuple $(x_1,...,x_N) \in X^N$. In previous works on collective dynamics the notion of transitivity has been deeply studied (see \cite{Banks2005,LiWaZha2006,Peris2005} and \cite{JarSanSan2020b}):
\begin{enumerate}[--]
	\item a dynamical system $(X,f)$ is called {\em topologically transitive} if given any pair of non-empty open subsets $U,V \subset X$ there exists some (and hence infinitely many) $n \in \NN$ such that $f^n(U) \cap V \neq \varnothing$.
\end{enumerate}
Also the notion of entropy (see \cite{KimChenJu2017,RiKimJu2023}) and many concepts of chaos (Devaney, Li-Yorke, distributional) have been considered (see \cite{JarSan2021_FS,JarSanSan2020a,MarPeRo2021}), but here we are just interested in that of $\Ac$-transitivity since it presents some similarities with the recurrence properties that we are about to introduce:
\begin{enumerate}[--]
	\item a collection of sets $\Ac \subset \Part(\NN_0)$ is called a {\em Furstenberg family} (or just a {\em family}) if it is hereditarily upward but $\varnothing \notin \Ac$, i.e.\ if given $A \in \Ac$ and $B \subset \NN_0$ with $A \subset B$ then $B \in \Ac$ but $\Ac \neq \Part(\NN_0)$;
	
	\item a dynamical system $(X,f)$ is called {\em topologically $\Ac$-transitive}, for a Furstenberg family $\Ac \subset \Part(\NN_0)$, if given any pair of non-empty open subsets $U,V \subset X$ we have that the {\em return set from $U$ to $V$}, which will be denoted by $\Nc_f(U,V) := \{ n \in \NN_0 \ ; \ f^n(U) \cap V \neq \varnothing \}$, belongs to the family $\Ac$.
\end{enumerate}
Typical explicit examples of this last definition arise when we consider: the family~$\Ac_{\infty}$, formed by the infinite subsets of $\NN_0$, so that {\em topological $\Ac_{\infty}$-transitivity} is exactly {\em topological transitivity}; the family of co-finite sets~$\Ac_{cf}$, whose respective {\em topological $\Ac_{cf}$-transitivity} notion is called {\em topological mixing}; or the family~$\Ac_{th}$, formed by the so-called {\em thick sets} (i.e.\ the sets $A \subset \NN_0$ such that for every $\ell \in \NN$ there is $n \in A$ with $[\![ n , n+\ell ]\!] \subset A$), and whose respective {\em topological $\Ac_{th}$-transitivity} notion coincides with the well-known property of {\em topological weak-mixing} (see for instance \cite[Theorem~1.54]{GrPe2011_book}).

Similar to the previous concepts but for recurrence, a dynamical system $(X,f)$ is called:
\begin{enumerate}[--]
	\item {\em topologically recurrent} if given any non-empty open subset $U \subset X$ there exists some (and hence infinitely many) $n \in \NN$ such that $f^n(U) \cap U \neq \varnothing$;
	
	\item {\em multiply recurrent} if given any positive integer $\ell \in \NN$ and any non-empty open subset $U \subset X$ there exists some (and hence infinitely many) $n \in \NN$ such that
	\[
	\bigcap_{0\leq j\leq \ell} f^{-jn}(U) = U \cap f^{-n}(U) \cap f^{-2n}(U) \cap \cdots \cap f^{-\ell n}(U) \neq \varnothing;
	\]
	
	\item {\em topologically $\Ac$-recurrent}, for a Furstenberg family $\Ac \subset \Part(\NN_0)$, if given any non-empty open subset $U \subset X$ we have that the return set $\Nc_f(U,U) = \{ n \in \NN_0 \ ; \ f^n(U) \cap U \neq \varnothing \}$, which from now on will be simply denoted by $\Nc_f(U) := \Nc_f(U,U)$, belongs to the family $\Ac$.
\end{enumerate}

\begin{remark}\label{Rem:nonwandering.VanDerWaerden}
	The previous notions of recurrence are important in Topological Dynamics. Indeed:
	\begin{enumerate}[(a)]
		\item The topologically recurrent dynamical systems have also been called {\em nonwandering systems} in the literature (see \cite[Chapter~1, Section~8]{Furstenberg1981_book}). However, in order to maintain a symmetry between this notion and the respective concept of ``topological transitivity'' already introduced, we will use the terminology of ``topological recurrence'' coined in the 2014 paper \cite{CoMaPa2014}.
		
		\item The multiply recurrent dynamical systems have been deeply studied in the literature (see for instance \cite[Chapter~2,~Section~2]{Furstenberg1981_book}), and they have recently been considered in the works~\cite{CarMur2022_MS,CoPa2012}, but also in the recent 2017 paper \cite{KwiLiOpYe2017} under the name of {\em Van der Waerden systems}.
		
		\item The notion of topological $\Ac$-recurrence, for a general Furstenberg family $\Ac \subset \Part(\NN_0)$, was recently defined in the 2022 paper \cite{BoGrLoPe2022} and its systematic study has been continued in \cite{AA2024}.
		
		\item As it happens for transitivity, the notion of ``topological recurrence'' can be expressed as that of ``topological $\Ac_{\infty}$-recurrence'' for the family $\Ac_{\infty}$ of infinite subsets of $\NN_0$. However, there is no family $\Ac \subset \Part(\NN_0)$ for which ``topological $\Ac$-recurrence'' equals ``multiple recurrence''. Indeed, there exist dynamical systems that even are topologically $\Ac_{cf}$-transitive (i.e.\ topologically mixing) while they are not multiply recurrent (see~\cite[Proposition~4.17]{KwiLiOpYe2017}).
	\end{enumerate}
\end{remark}

In order to give a unified treatment to the previous properties we will use the following definition:

\begin{definition}\label{Def:(l,A)-rec}
	Let $\ell \in \NN$ be a positive integer and let $\Ac \subset \Part(\NN_0)$ be a Furstenberg family. We will say that a dynamical system $(X,f)$ is {\em topologically $(\ell,\Ac)$-recurrent} if given any non-empty open subset $U \subset X$ we have that the {\em $\ell$-return set from $U$ to itself}, which is defined as
	\[
	\Nc_f^{\ell}(U) := \left\{ n \in \NN_0 \ ; \ \bigcap_{0\leq j\leq \ell} f^{-jn}(U) = U \cap f^{-n}(U) \cap f^{-2n}(U) \cap \cdots \cap f^{-\ell n}(U) \neq \varnothing \right\},
	\]
	belongs to the family $\Ac$.
\end{definition}

\begin{remark}\label{Rem:(l,A)-rec}
	Using Definition~\ref{Def:(l,A)-rec} we can give a unified treatment of the three recurrence properties previously introduced. Indeed, since:
	\begin{enumerate}[--]
		\item {\em topological recurrence} coincides with {\em topological $(1,\Ac_{\infty})$-recurrence};
		
		\item {\em multiple recurrence} coincides with being {\em topologically $(\ell,\Ac_{\infty})$-recurrent} for every $\ell \in \NN$;
		
		\item and {\em topological $\Ac$-recurrence} coincides with {\em topological $(1,\Ac)$-recurrence} for the respective family $\Ac$;
	\end{enumerate}
	then the {\em topological $(\ell,\Ac)$-recurrence}'s results proved for all $\ell \in \NN$ and every family $\Ac \subset \Part(\NN_0)$ will be directly valid for topological recurrence, multiple recurrence and topological $\Ac$-recurrence.
\end{remark}

\subsection[The hyperspaces K(X) and F(X)]{The hyperspaces $\Kc(X)$ and $\Fc(X)$}

For a topological space $X$ we will denote its {\em hyperspace of non-empty compact subsets} by $\Kc(X)$. Given a continuous map $f:X\longrightarrow X$ we will denote by $\com{f}:\Kc(X)\longrightarrow\Kc(X)$ its {\em hyperextension}, defined as
\[
\com{f}(K) := f(K) \quad \text{ for each } K \in \Kc(X),
\]
where $f(K) := \{ f(x) \ ; \ x \in K \}$ as usual. It is well-known that the map $\com{f}$ is continuous with respect to the so-called {\em Vietoris topology} on $\Kc(X)$, whose basic open sets are the sets of the form
\[
\Vc(U_1,...,U_N) :=  \left\{ K \in \Kc(X) \ ; \ K \subset \bigcup_{j=1}^N U_j \text{ and } K\cap U_j\neq\varnothing \text{ for all } 1\leq j\leq N \right\},
\]
where $N \in \NN$ and $U_1,...,U_N$ are non-empty open subsets of $X$. Moreover, if the space $X$ is metrizable by a metric $d:X\times X\longrightarrow[0,+\infty[$, then the Vietoris topology is equivalent to that induced by the so-called {\em Hausdorff metric} $d_H:\Kc(X)\times\Kc(X)\longrightarrow[0,+\infty[$ defined by
\[
d_H(K_1,K_2) := \max\left\{ \max_{x_1 \in K_1} d(x_1,K_2) , \max_{x_2 \in K_2} d(x_2,K_1) \right\} \quad \text{ for each pair } K_1,K_2 \in \Kc(X).
\]
We will denote by $\Bc_H(K,\eps)$ the open ball centered at $K \in \Kc(X)$ and of radius $\eps>0$ for the metric $d_H$, and along the paper we will need the following well-known fact:

\begin{proposition}\label{Pro:Hausdorff}
	Let $(X,d)$ be a metric space and let $A,B,C,D \in \Kc(X)$. Then we have that
	\[
	d_H(A\cup B, C \cup D) \leq \max\{ d_H(A,C) , d_H(B,D) \}.
	\]
\end{proposition}

We refer the reader to \cite{IllNad1999} for a detailed study of hyperspaces of compact sets.

A {\em fuzzy~set} on a topological space $X$ is a function $u:X\longrightarrow[0,1]$, where the value $u(x) \in [0,1]$ denotes the degree of membership of the point $x$ in the fuzzy set $u$. Following the notation used in the literature we will write
\[
u_{\alpha} := \{ x \in X \ ; \ u(x) \geq \alpha \} \text{ for each } \alpha \in \ ]0,1] \quad \text{ and } \quad u_0 := \cl{\{ x \in X \ ; \ u(x)>0 \}} = \cl{\bigcup_{\alpha\in]0,1]} u_{\alpha}}.
\]
In this paper we will work with the {\em hyperspace of normal fuzzy sets}, i.e.\ the space formed by the fuzzy sets that are upper-semicontinuous functions and such that $u_0$ is compact and $u_1$ is non-empty. This space will be denoted by $\Fc(X)$, and it is worth noticing that given any non-empty compact set $K \in \Kc(X)$ then the {\em characteristic function} on $K$, denoted by $\chi_K:X\longrightarrow[0,1]$, belongs to $\Fc(X)$.

Given a (not necessarily continuous) map $f:X\longrightarrow X$ we will denote by $\fuz{f}$ its {\em Zadeh extension}, which transforms each (not necessarily normal) fuzzy set $u:X\longrightarrow[0,1]$ into the following fuzzy set
\begin{equation}\label{eq:Zadeh}
\fuz{f}(u):X\longrightarrow[0,1] \quad \text{ where } \quad \fuz{f}(u)(x) :=
\left\{
\begin{array}{lcc}
	\sup\{ u(y) \ ; \ y \in f^{-1}(x) \}, & \text{ if } f^{-1}(x) \neq \varnothing, \\[7.5pt]
	0, & \text{ if } f^{-1}(x) = \varnothing.
\end{array}
\right.
\end{equation}
It follows from \cite[Propositions~3.1 and 4.9]{JarSanSan2020a} that, when $X$ is a Hausdorff topological space and the map $f:X\longrightarrow X$ is continuous, then the respective Zadeh extension $\fuz{f}$ (defined as in \eqref{eq:Zadeh} and restricted to the space of normal fuzzy sets $\Fc(X)$) is a well-defined self-map of $\Fc(X)$. In order to talk about continuity for $\fuz{f}:\Fc(X)\longrightarrow\Fc(X)$ we must endow $\Fc(X)$ with some topology, and this forces us to consider metric spaces. Following \cite{JarSan2021_FS,KimChenJu2017,MarPeRo2021,RiKimJu2023}, given a metric space $(X,d)$ we will consider:
\begin{enumerate}[--]
	\item the {\em supremum metric} $d_{\infty}:\Fc(X)\times\Fc(X)\longrightarrow[0,+\infty[$ defined for each pair $u,v \in \Fc(X)$ as
	\[
	d_{\infty}(u,v) := \sup_{\alpha \in [0,1]} d_H(u_{\alpha},v_{\alpha}),
	\]
	where $d_H$ denotes the already introduced Hausdorff metric on $\Kc(X)$. We denote by $\Bc_{\infty}(u,\eps)$ the open ball centered at $u \in \Fc(X)$ and of radius $\eps>0$ for the metric $d_{\infty}$, and we will denote by $\Fc_{\infty}(X)$ the topological space $(\Fc(X),d_{\infty})$ for short, which is a metric space that is non-separable as soon as~$X$ has more than one point (see Section~\ref{Sec_5:complete} for more on the non-separability of $\Fc_{\infty}(X)$).
	
	\item the {\em Skorokhod metric} $d_{0}:\Fc(X)\times\Fc(X)\longrightarrow[0,+\infty[$ defined for each pair $u,v \in \Fc(X)$ as
	\[
	d_0(u,v) := \inf\left\{ \eps \ ; \ \text{there is } \xi \in \Tc \text{ such that } \sup_{\alpha\in[0,1]} |\xi(\alpha)-\alpha| \leq \eps \text{ and } d_{\infty}(u,\xi\circ v) \leq \eps \right\},
	\]
	where $\Tc$ is the set of strictly increasing homeomorphisms of the unit interval $\xi:[0,1]\longrightarrow[0,1]$. We denote by $\Bc_{0}(u,\eps)$ the open ball centered at $u \in \Fc(X)$ and of radius $\eps>0$ for the metric $d_{0}$, and we will denote by $\Fc_{0}(X)$ the topological space $(\Fc(X),d_{0})$ for short, which is a metric space.
\end{enumerate}

It was showed in \cite[Theorems~4.7~and~4.8]{JarSanSan2020a} that $\fuz{f}$ is continuous in $\Fc_{\infty}(X)$ and $\Fc_{0}(X)$ if and only if the map $f$ is continuous on the metric space $(X,d)$. Let us include some important well-known properties of the Zadeh extension $\fuz{f}$ that we will use without citing them (see \cite{ChaRo2008,JarSan2021_FS,JarSanSan2020a}):

\begin{proposition}\label{Pro:Zadeh}
	Let $(X,f)$ be a dynamical system where $(X,d)$ is a metric space, $u,v \in \Fc(X)$, $\alpha \in [0,1]$, $n \in \NN_0$ and $K \in \Kc(X)$. Then:
	\begin{enumerate}[{\em(a)}]
		\item $\left[ \fuz{f}(u) \right]_{\alpha} = f(u_{\alpha})$;
		
		\item $\left( \fuz{f} \right)^n = \widehat{f^n}$;
		
		\item $\fuz{f}\left( \chi_K \right) = \chi_{\com{f}(K)} = \chi_{f(K)}$;
		
		\item $d_{0}(u,v) \leq d_{\infty}(u,v)$ but for characteristic functions we have that $d_{0}(u,\chi_K)=d_{\infty}(u,\chi_K)$.
	\end{enumerate}
\end{proposition}

Also, the following well-known normal fuzzy sets lemma will be very useful (see \cite{JarSan2021_FS,JarSanSan2020a,JarSanSan2020b,MarPeRo2021}):

\begin{lemma}\label{Lem:eps.pisos}
	Let $(X,d)$ be a metric space. Given any $u \in \Fc(X)$ and $\eps>0$ there exist numbers $0 = \alpha_0 < \alpha_1 < \alpha_2 < ... < \alpha_N = 1$ such that $d_H(u_{\alpha}, u_{\alpha_{i+1}})<\eps$ for each $\alpha \in ]\alpha_i, \alpha_{i+1}]$ and $1\leq i\leq N-1$. In particular, since $d_H(u_{\alpha}, u_{\alpha_1})<\eps$ for every $\alpha \in \ ]0, \alpha_1]$, we also have that $d_H(u_0,u_{\alpha_1})\leq\eps$.
\end{lemma}

\section[Recurrence in K(X)]{Recurrence in $\Kc(X)$}\label{Sec_3:K(X)}

To study recurrence in $\Kc(X)$, the following general result will directly imply Corollary~\ref{Cor:K.rec} below, in which topological recurrence, multiple recurrence and topological $\Ac$-recurrence are included.

\begin{theorem}\label{The:K.rec}
	Let $\ell \in \NN$ be a positive integer and let $\Ac \subset \Part(\NN_0)$ be a Furstenberg family. Then, given a continuous map $f:X\longrightarrow X$ on a topological space $X$, the following statements are equivalent:
	\begin{enumerate}[{\em(i)}]
		\item $f_{(N)}:X^N\longrightarrow X^N$ is topologically $(\ell,\Ac)$-recurrent for every $N \in \NN$;
		
		\item $\com{f}_{(N)}:\Kc(X)^N\longrightarrow \Kc(X)^N$ is topologically $(\ell,\Ac)$-recurrent for every $N \in \NN$;
		
		\item $\com{f}:\Kc(X)\longrightarrow \Kc(X)$ is topologically $(\ell,\Ac)$-recurrent.
	\end{enumerate}
\end{theorem}
\begin{proof}
	(i) $\Rightarrow$ (ii): Given $N \in \NN$ basic Vietoris-open sets, denoted by $\Vc(U_1^j,...,U_{k_j}^j)$ for $1\leq j\leq N$, set
	\[
	A := \bigcap_{1\leq j\leq N} \Nc_{\com{f}}^{\ell}\left( \Vc(U_1^j,...,U_{k_j}^j) \right) = \bigcap_{1\leq j\leq N} \left\{ n \in \NN_0 \ ; \ \bigcap_{0\leq i\leq \ell} \com{f}^{-in}\left( \Vc(U_1^j,...,U_{k_j}^j) \right) \neq \varnothing \right\}.
	\]
	We have to show that the set $A$ belongs to the family $\Ac$. Indeed, by assumption we have that $f_{(M)}$ is topologically $(\ell,\Ac)$-recurrent for the positive integer $M:=\sum_{j=1}^n k_j \in \NN$, so that the set
	\[
	B := \Nc_{f_{(M)}}^{\ell}\left( \prod_{1\leq j\leq N}^{1\leq k\leq k_j} U_k^j \right) = \bigcap_{1\leq j\leq N}^{1\leq k\leq k_j} \left\{ n \in \NN_0 \ ; \ \bigcap_{0\leq i\leq \ell} \com{f}^{-in}( U_k^j ) \neq \varnothing \right\}
	\]
	does belong to $\Ac$. Given now any arbitrary but fixed $n \in B$ we can select points
	\[
	x_k^j \in U_k^j \cap f^{-n}(U_k^j) \cap \cdots \cap f^{-\ell n}(U_k^j) \quad \text{ for each } 1\leq j\leq N \text{ and } 1\leq k\leq k_j,
	\]
	and define $K_j := \{ x_1^j, x_2^j, ..., x_{k_j}^j \}$ for each $1\leq j\leq N$. It is not difficult to check that
	\[
	K_j \in \Vc(U_1^j,...,U_{k_j}^j) \cap \com{f}^{-n}\left( \Vc(U_1^j,...,U_{k_j}^j) \right) \cap \cdots \cap \com{f}^{-\ell n}\left( \Vc(U_1^j,...,U_{k_j}^j) \right) \quad \text{ for each } 1\leq j \leq N,
	\]
	so that $n \in A$. Since $n \in B$ was arbitrary we deduce that $B \subset A$ and hence $A \in \Ac$ as we had to show.
	
	(ii) $\Rightarrow$ (iii): This is obvious.
	
	(iii) $\Rightarrow$ (i): We proceed by induction on $N \in \NN$. For $N=1$, given any arbitrary but fixed non-empty open subset $U \subset X$ then the topological $(\ell,\Ac)$-recurrence of $\com{f}$ shows that
	\[
	C := \Nc_{\com{f}}^{\ell}\left( \Vc(U) \right) = \left\{ n \in \NN_0 \ ; \ \bigcap_{0\leq i\leq \ell} \com{f}^{-in}\left( \Vc(U) \right) \neq \varnothing \right\} \in \Ac.
	\]
	Hence, given any arbitrary but fixed $n \in C$ we can find some $K_n \in \Vc(U) \cap \com{f}^{-n}(\Vc(U)) \cap \cdots \cap \com{f}^{-\ell n}(\Vc(U))$. In particular, given any point $x \in K_n$ we have that $x \in U \cap f^{-n}(U) \cap \cdots \cap f^{-\ell n}(U)$. Since $n \in C$ was arbitrary we deduce that
	\[
	C \subset \Nc_f^{\ell}(U) \quad \text{ and hence } \quad \Nc_f^{\ell}(U) \in \Ac.
	\]
	This implies that $f$ is topologically $(\ell,\Ac)$-recurrent by the arbitrariness of the open set $U$.
	
	Assume now that $f_{(N-1)}$ is topologically $(\ell,\Ac)$-recurrent for some $N>1$ and consider $N$ arbitrary but fixed non-empty open subsets $U_1,...,U_N \subset X$. In order to finally conclude that the map $f_{(N)}$ is topologically $(\ell,\Ac)$-recurrent we must show that the set
	\[
	D := \Nc_{f_{(N)}}^{\ell}\left( \prod_{1\leq j\leq N} U_j \right) = \bigcap_{1\leq j\leq N} \left\{ n \in \NN_0 \ ; \ \bigcap_{0\leq i\leq \ell} f^{-in}( U_j ) \neq \varnothing \right\}
	\]
	belongs to $\Ac$. Using the topological $(\ell,\Ac)$-recurrence of $\com{f}$ we have that
	\[
	E := \Nc_{\com{f}}^{\ell}\left( \Vc(U_1,...,U_N) \right) = \left\{ n \in \NN_0 \ ; \ \bigcap_{0\leq i\leq \ell} \com{f}^{-in}\left( \Vc(U_1,...,U_N) \right) \neq \varnothing \right\} \in \Ac.
	\]
	From now on we have two possibilities:
	\begin{enumerate}[--]
		\item \textbf{Case 1}: \textit{For each $n \in E$ there exists a compact set
		\[
		K_n \in \Vc(U_1,...,U_N) \cap \com{f}^{-n}(\Vc(U_1,...,U_N)) \cap \cdots \cap \com{f}^{-\ell n}(\Vc(U_1,...,U_N))
		\]
		such that $K_n \cap U_j \cap f^{-n}(U_j) \cap \cdots \cap f^{-\ell n}(U_j) \neq \varnothing$ for every $1\leq j\leq N$}. In this case it follows that $E \subset D$, and hence that $D \in \Ac$ by the hereditarily upward condition of $\Ac$, as we had to show.
		
		\item \textbf{Case 2}: \textit{There exists some $n_0 \in E$ such that for every compact set
		\begin{equation}\label{eq:K1}
			K \in \Vc(U_1,...,U_N) \cap \com{f}^{-n_0}(\Vc(U_1,...,U_N)) \cap \cdots \cap \com{f}^{-\ell n_0}(\Vc(U_1,...,U_N)),
		\end{equation}
		then there exists some $1\leq j_0\leq N$ (which depends on the compact set $K$ selected) for which}
		\begin{equation}\label{eq:K2}
			K \cap U_{j_0} \cap f^{-n_0}(U_{j_0}) \cap \cdots \cap f^{-\ell n_0}(U_{j_0}) = \varnothing.
		\end{equation}
		In this case fix $K$ and $j_0$ as described above and note that, by \eqref{eq:K1} and \eqref{eq:K2}, then there exists indexes $1\leq k_0\neq j_0 \leq N$ and $0\leq i_0\leq \ell$ such that $f^{i_0n_0}(K \cap U_{j_0}) \cap U_{k_0} \neq \varnothing$. Hence, the set $U_0 := U_{j_0} \cap f^{-i_0n_0}(U_{k_0})$ is a non-empty open subset of $X$. Using now that $f_{(N-1)}$ is topologically $(\ell,\Ac)$-recurrent, by the induction hypothesis, we have that
		\[
		F := \Nc_{f_{(N-1)}}^{\ell}\left( \prod_{0\leq j\leq N}^{k_0 \neq j \neq j_0} U_j \right) = \bigcap_{0\leq j\leq N}^{k_0 \neq j \neq j_0} \left\{ n \in \NN_0 \ ; \ \bigcap_{0\leq i\leq \ell} f^{-in}( U_j ) \neq \varnothing \right\} \in \Ac.
		\]
		Since $U_0 \subset U_{j_0}$ we also have that $F \subset \Nc_f^{\ell}(U_{j_0})$. Moreover, given any arbitrary but fixed integer $n \in F$ then we can select a point $x \in U_0 \cap f^{-n}(U_0) \cap \cdots \cap f^{-\ell n}(U_0)$, and considering $y := f^{i_0n_0}(x)$ we have that
		\[
		f^{in}(y) = f^{in}(f^{i_0n_0}(x)) = f^{i_0n_0}(f^{in}(x)) \in f^{i_0n_0}(U_0) \subset U_{k_0} \quad \text{ for every } 0\leq i\leq \ell.
		\]
		This implies that $y \in U_{k_0} \cap f^{-n}(U_{k_0}) \cap \cdots \cap f^{-\ell n}(U_{k_0})$ and hence that $n \in \Nc_f^{\ell}(U_{k_0})$. The arbitrariness of $n \in F$ shows that $F \subset \Nc_f^{\ell}(U_{k_0})$, and we deduce that
		\[
		F \subset \bigcap_{1\leq j\leq N} \Nc_f^{\ell}(U_j) = \bigcap_{1\leq j\leq N} \left\{ n \in \NN_0 \ ; \ \bigcap_{0\leq i\leq \ell} f^{-in}( U_j ) \neq \varnothing \right\} = D.
		\]
		By the hereditarily upward condition of $\Ac$ we have that $D \in \Ac$, as we had to show.\qedhere
	\end{enumerate}
\end{proof}

\begin{remark}\label{Rem:K.rec}
	The statement and proof of Theorem~\ref{The:K.rec} admit several remarks:
	\begin{enumerate}[(a)]
		\item Given any $N \in \NN$, the map $\com{f}_{(N)}:\Kc(X)^N\longrightarrow \Kc(X)^N$ included in statement (ii) of Theorem~\ref{The:K.rec} is not necessarily equal to the map
		\[
		\com{f_{(N)}}:\Kc(X^N)\longrightarrow \Kc(X^N),
		\]
		because the spaces $\Kc(X)^N$ and $\Kc(X^N)$ are not necessarily equal. Indeed, since the product of compact sets is again compact we always have the following natural (and easy to check) continuous injection from the product space $\Kc(X)^N$ to the hyperspace $\Kc(X^N)$:
		\[
		\iota:\Kc(X)^N\longrightarrow \Kc(X^N) \quad \text{ with } \iota\left( (K_1,K_2,...,K_N) \right) := K_1\times K_2\times\cdots\times K_N.
		\]
		However, the space $\Kc(X^N)$ can be much bigger than $\Kc(X)^N$: consider for instance $X=\RR$ with the usual topology and note that the closed ball $\{ (x,y) \in \RR^2 \ ; \ x^2+y^2=1 \}$ belongs to $\Kc(\RR^2)$ while it cannot be expressed as an element of $\Kc(\RR)^2$.
		
		\item Despite the previous comment, Theorem~\ref{The:K.rec} itself can be used to obtain (in its statement) the following two extra equivalent conditions:
		\begin{enumerate}
			\item[(ii')] {\em $\com{f_{(N)}}:\Kc(X^N)\longrightarrow \Kc(X^N)$ is topologically $(\ell,\Ac)$-recurrent for every $N \in \NN$};
			
			\item[(iii')] {\em $\com{f_{(N)}}:\Kc(X^N)\longrightarrow \Kc(X^N)$ is topologically $(\ell,\Ac)$-recurrent for some $N \in \NN$}.
		\end{enumerate}
		This follows by applying Theorem~\ref{The:K.rec} itself to the map $f_{(N)}$ for a particular $N \in \NN$ instead than applying it to $f$, and then by noticing that:
		\begin{enumerate}[--]
			\item we have the equality $(f_{(N)})_{(M)} = f_{(N \cdot M)}$ for every $M \in \NN$;
			
			\item if the map $f_{(M)}$ is topologically $(\ell,\Ac)$-recurrent for some $M \in \NN$, then the map $f_{(J)}$ is also topologically $(\ell,\Ac)$-recurrent for every $1\leq J\leq M$.
		\end{enumerate}
		
		\item If the Furstenberg family $\Ac \subset \Part(\NN_0)$ considered in the statement of Theorem~\ref{The:K.rec} is not a filter, then the $N$-fold direct product hypothesis for every $N \in \NN$ in statement (i) of Theorem~\ref{The:K.rec} is necessary. Indeed, for the case of usual topological recurrence (i.e.\ topological $(1,\Ac_{\infty})$-recurrence), it was shown in \cite[Example~4]{Banks1999} and \cite[Theorem~3.2]{GriLoPe2025_AMP} that given any $N \in \NN$ there exist many dynamical systems $(X,f)$ such that:
		\begin{enumerate}[--]
			\item the map $f_{(N)}:X^N\longrightarrow X^N$ is topologically recurrent (and even multiply recurrent);
			
			\item but such that the next product $f_{(N+1)}:X^{N+1}\longrightarrow X^{N+1}$ is not topologically recurrent.
		\end{enumerate}
	\end{enumerate}
\end{remark}

Using Remark~\ref{Rem:(l,A)-rec} we can now state some interesting consequences of Theorem~\ref{The:K.rec}. In particular, we obtain the ``topological recurrence'' analogue to the respective ``topological transitivity'' result independently obtained by Banks~\cite{Banks2005}, Liao et al.~\cite{LiWaZha2006} and Peris~\cite{Peris2005}:

\begin{corollary}\label{Cor:K.rec}
	Let $f:X\longrightarrow X$ be a continuous map on a topological space $X$. Then:
	\begin{enumerate}[{\em(a)}]
		\item The following statements are equivalent:
		\begin{enumerate}[{\em(i)}]
			\item $f_{(N)}:X^N\longrightarrow X^N$ is topologically recurrent for every $N \in \NN$;
			
			\item $\com{f}_{(N)}:\Kc(X)^N\longrightarrow \Kc(X)^N$ is topologically recurrent for every $N \in \NN$;
			
			\item $\com{f}:\Kc(X)\longrightarrow \Kc(X)$ is topologically recurrent.
		\end{enumerate}
	
		\item The following statements are equivalent:
		\begin{enumerate}[{\em(i)}]
			\item $f_{(N)}:X^N\longrightarrow X^N$ is multiply recurrent for every $N \in \NN$;
			
			\item $\com{f}_{(N)}:\Kc(X)^N\longrightarrow \Kc(X)^N$ is multiply recurrent for every $N \in \NN$;
			
			\item $\com{f}:\Kc(X)\longrightarrow \Kc(X)$ is multiply recurrent.
		\end{enumerate}
	
		\item Given any Furstenberg family $\Ac \subset \Part(\NN_0)$, the following statements are equivalent:
		\begin{enumerate}[{\em(i)}]
			\item $f_{(N)}:X^N\longrightarrow X^N$ is topologically $\Ac$-recurrent for every $N \in \NN$;
			
			\item $\com{f}_{(N)}:\Kc(X)^N\longrightarrow \Kc(X)^N$ is topologically $\Ac$-recurrent for every $N \in \NN$;
			
			\item $\com{f}:\Kc(X)\longrightarrow \Kc(X)$ is topologically $\Ac$-recurrent.
		\end{enumerate}
	\end{enumerate}
\end{corollary}
\begin{proof}
	This is a direct consequence of Remark~\ref{Rem:(l,A)-rec} and Theorem~\ref{The:K.rec}.
\end{proof}

In the next section we will see that Theorem~\ref{The:K.rec} is also very useful to study the recurrence-kind properties of the fuzzy system $(\Fc(X),\fuz{f})$.

\section[Recurrence in F(X)]{Recurrence in $\Fc(X)$}\label{Sec_4:F(X)}

To consider $\Fc(X)$ we will assume that $X$ is a metric space as argued in Section~\ref{Sec_2:notation}. The following general result can be compared with Theorem~\ref{The:K.rec} and will directly imply the particular cases of topological recurrence, multiple recurrence and topological $\Ac$-recurrence (see Corollary~\ref{Cor:F.rec} below).

\begin{theorem}\label{The:F.rec}
	Let $\ell \in \NN$ be a positive integer and let $\Ac \subset \Part(\NN_0)$ be a Furstenberg family. Then, given a continuous map $f:X\longrightarrow X$ on a metric space $(X,d)$, the following statements are equivalent:
	\begin{enumerate}[{\em(i)}]
		\item $f_{(N)}:X^N\longrightarrow X^N$ is topologically $(\ell,\Ac)$-recurrent for every $N \in \NN$;
		
		\item $\com{f}:\Kc(X)\longrightarrow \Kc(X)$ is topologically $(\ell,\Ac)$-recurrent;
		
		\item $\fuz{f}:\Fc_{\infty}(X)\longrightarrow \Fc_{\infty}(X)$ is topologically $(\ell,\Ac)$-recurrent;
		
		\item $\fuz{f}:\Fc_{0}(X)\longrightarrow \Fc_{0}(X)$ is topologically $(\ell,\Ac)$-recurrent.
	\end{enumerate}
\end{theorem}
\begin{proof}
	(i) $\Leftrightarrow$ (ii): This equivalence is showed in Theorem~\ref{The:K.rec}.
	
	(ii) $\Rightarrow$ (iii): Fix $u \in \Fc(X)$ and $\eps>0$. We must show that the set
	\[
	A := \Nc_{\fuz{f}}^{\ell}\left( \Bc_{\infty}(u,\eps) \right) = \left\{ n \in \NN_0 \ ; \ \bigcap_{0\leq j\leq \ell} \fuz{f}^{-jn}\left( \Bc_{\infty}(u,\eps) \right) \neq \varnothing \right\}
	\]
	belongs to $\Ac$. By Lemma~\ref{Lem:eps.pisos} there exist numbers $0 = \alpha_0 < \alpha_1 < \alpha_2 < ... < \alpha_N = 1$ such that
	\begin{equation}\label{eq:eps.pisos}
		d_H(u_{\alpha}, u_{\alpha_{i+1}})<\tfrac{\eps}{2} \text{ for each } \alpha \in \ ]\alpha_i, \alpha_{i+1}] \text{ with } 1\leq i\leq N-1 \quad \text{ and also } \quad d_H(u_0,u_{\alpha_1})\leq\tfrac{\eps}{2}.
	\end{equation}
	By Theorem~\ref{The:K.rec}, the dynamical system $(\Kc(X)^N,\com{f}_{(N)})$ is topologically $(\ell,\Ac)$-recurrent so that
	\[
	B := \Nc_{\com{f}_{(N)}}^{\ell}\left( \prod_{1\leq i\leq N} \Bc_H(u_{\alpha_i},\tfrac{\eps}{2}) \right) = \bigcap_{1\leq i\leq N} \left\{ n \in \NN_0 \ ; \ \bigcap_{0\leq j\leq \ell} \com{f}^{-jn}\left( \Bc_H(u_{\alpha_i},\tfrac{\eps}{2}) \right) \right\} \in \Ac.
	\]
	Thus, given any arbitrary but fixed $n \in B$ there exist compact sets $K_1,...,K_N \in \Kc(X)$ such that $K_i \in \bigcap_{0\leq j\leq \ell} \com{f}^{-jn}\left( \Bc_H(u_{\alpha_i},\tfrac{\eps}{2}) \right)$ for each $1\leq i\leq N$, and in particular we have that
	\begin{equation}\label{eq:d_H.f^{jn}(K_i)}
		\max_{0\leq j\leq \ell} \left\{ d_H\left(u_{\alpha_i},f^{jn}(K_i)\right) \ ; \ 1\leq i\leq N \right\} < \tfrac{\eps}{2}.
	\end{equation}
	Consider now the upper-semicontinuous function
	\[
	v := \max_{1\leq i\leq N} \left( \alpha_i \cdot \chi_{K_i} \right) : X \longrightarrow [0,1],
	\]
	which is a fuzzy set fulfilling that $v_{\alpha} = \bigcup \left\{ K_i \ ; \ 1\leq i\leq N \text{ with } \alpha_i \geq \alpha \right\}$ for each $\alpha \in [0,1]$, and in particular
	\[
	v_{\alpha_i} = \bigcup_{l \geq i} K_l \quad \text{ for each } 1\leq i\leq N.
	\]
	It follows that $v$ is a normal fuzzy set, i.e.\ $v \in \Fc(X)$, and by Proposition~\ref{Pro:Hausdorff} we also have that
	\begin{align}\label{eq:d_H.f^{jn}(v_i)}
		d_H\left(u_{\alpha_i},f^{jn}(v_{\alpha_i})\right) &= d_H\left( \bigcup_{l\geq i} u_{\alpha_l}, f^{jn}\left( \bigcup_{l\geq i} K_l \right) \right) = d_H\left( \bigcup_{l\geq i} u_{\alpha_l}, \bigcup_{l\geq i} f^{jn}(K_l)\right) \\[7.5pt]
		&\leq \max_{0\leq j\leq \ell} \left\{ d_H\left( u_{\alpha_l} , f^{jn}(K_l) \right) \ ; \ 1\leq l\leq N \right\} \overset{\eqref{eq:d_H.f^{jn}(K_i)}}{<} \frac{\eps}{2} \quad \text{ for all } 1\leq i\leq N \text{ and } 0\leq j\leq \ell. \nonumber
	\end{align}
	Moreover, given any $0\leq j\leq \ell$ and any $\alpha \in [0,1]$ we have that
	\begin{align*}
		d_H\left( u_{\alpha} , \left[\fuz{f}^{jn}(v)\right]_{\alpha} \right) &= d_H\left( u_{\alpha} , f^{jn}(v_{\alpha}) \right)
		= \left\{
		\begin{array}{lcc}
			d_H\left( u_{\alpha} , f^{jn}(v_{\alpha_1}) \right), & \text{ if } \alpha \in [0,\alpha_1] \hspace{2.95cm} \\[5pt]
			d_H\left( u_{\alpha} , f^{jn}(v_{\alpha_{i+1}}) \right), & \text{ if } \alpha \in ]\alpha_i,\alpha_{i+1}], 1\leq i\leq N-1
		\end{array}
		\right\} \\[7.5pt]
		&\leq \left\{
		\begin{array}{lcc}
			d_H\left( u_{\alpha} , u_{\alpha_1} \right) + d_H\left( u_{\alpha_1} , f^{jn}(v_{\alpha_1}) \right), & \text{ if } \alpha \in [0,\alpha_1] \hspace{2.95cm} \\[5pt]
			d_H\left( u_{\alpha} , u_{\alpha_{i+1}} \right) + d_H\left( u_{\alpha_{i+1}} , f^{jn}(v_{\alpha_{i+1}}) \right), & \text{ if } \alpha \in ]\alpha_i,\alpha_{i+1}], 1\leq i\leq N-1
		\end{array}
		\right\} \\[7.5pt]
		&\overset{\eqref{eq:eps.pisos}}{\leq} \tfrac{\eps}{2} + \max_{0\leq j\leq \ell} \left\{ d_H\left(u_{\alpha_i},f^{jn}(v_{\alpha_i})\right) \ ; \ 1\leq i\leq N \right\} \overset{\eqref{eq:d_H.f^{jn}(v_i)}}{<} \eps,
	\end{align*}
	which finally implies that $v \in \Bc_{\infty}(u,\eps) \cap \fuz{f}^{-n}\left( \Bc_{\infty}(u,\eps) \right) \cap \cdots \cap \fuz{f}^{-\ell n}\left( \Bc_{\infty}(u,\eps) \right)$. We deduce that $n \in A$ and, since $n \in B$ was arbitrary, we have that $B \subset A$ and hence that $A \in \Ac$ as we had to show.
	
	(iii) $\Rightarrow$ (iv): Trivial since the topology induced by $d_{\infty}$ in $\Fc(X)$ is finer than that induced by $d_{0}$.
	
	(iv) $\Rightarrow$ (ii): Fix $K \in \Kc(X)$ and $\eps>0$. We must show that the set
	\[
	C := \Nc_{\com{f}}^{\ell}\left( \Bc_H(K,\eps) \right) = \left\{ n \in \NN_0 \ ; \ \bigcap_{0\leq j\leq \ell} \com{f}^{-jn}\left( \Bc_H(K,\eps) \right) \neq \varnothing \right\}
	\]
	belongs to $\Ac$. Define $u := \chi_K$, which belongs to $\Fc(X)$, and since the dynamical system $(\Fc_{0}(X),\fuz{f})$ is assumed to be topologically $(\ell,\Ac)$-recurrent we have that
	\[
	D := \Nc_{\fuz{f}}^{\ell}\left( \Bc_{0}(u,\eps) \right) = \left\{ n \in \NN_0 \ ; \ \bigcap_{0\leq j\leq \ell} \fuz{f}^{-jn}\left( \Bc_{0}(u,\eps) \right) \neq \varnothing \right\} \in \Ac.
	\]
	Thus, given any arbitrary but fixed $n \in D$ there exists a normal fuzzy set $v \in \bigcap_{0\leq j\leq \ell} \fuz{f}^{-jn}\left( \Bc_{0}(u,\eps) \right)$, and in particular we have by Proposition~\ref{Pro:Zadeh} that
	\[
	d_{\infty}\left( u , \fuz{f}^{jn}(v) \right) = d_{\infty}\left( \chi_K , \fuz{f}^{jn}(v) \right) = d_{0}\left( \chi_K , \fuz{f}^{jn}(v) \right) < \eps \quad \text{ for all } 0\leq j\leq \ell,
	\]
	which implies that
	\[
	d_H\left( K , f^{jn}(v_1) \right) = d_H\left( K , [\fuz{f}^{jn}(v)]_1 \right) \leq d_{\infty}\left( u , \fuz{f}^{jn}(v) \right) < \eps  \quad \text{ for all } 0\leq j\leq \ell.
	\]
	Thus, the compact set $L:=v_1 \in \Kc(X)$ fulfills that
	\[
	L \in \Bc_H(K,\eps) \cap \com{f}^{-n}\left( \Bc_H(K,\eps) \right) \cap \cdots \cap \com{f}^{-\ell n}\left( \Bc_H(K,\eps) \right),
	\]
	so that $n \in C$. The arbitrariness of $n \in D$ now implies that $D \subset C$ and by the hereditarily upward condition of $\Ac$ we deduce that $C \in \Ac$, as we had to show.
\end{proof}

Again by Remark~\ref{Rem:(l,A)-rec}, and as we did in Corollary~\ref{Cor:K.rec}, we can get interesting recurrence-kind consequences of Theorem~\ref{The:F.rec}. In particular, we obtain the ``topological recurrence'' analogue to the respective ``topological transitivity'' result recently obtained in \cite{JarSanSan2020b}:

\begin{corollary}\label{Cor:F.rec}
	Let $f:X\longrightarrow X$ be a continuous map on a metric space $(X,d)$. Then:
	\begin{enumerate}[{\em(a)}]
		\item The following statements are equivalent:
		\begin{enumerate}[{\em(i)}]
			\item $f_{(N)}:X^N\longrightarrow X^N$ is topologically recurrent for every $N \in \NN$;
			
			\item $\com{f}:\Kc(X)\longrightarrow \Kc(X)$ is topologically recurrent;
			
			\item $\fuz{f}:\Fc_{\infty}(X)\longrightarrow \Fc_{\infty}(X)$ is topologically recurrent;
			
			\item $\fuz{f}:\Fc_{0}(X)\longrightarrow \Fc_{0}(X)$ is topologically recurrent.
		\end{enumerate}
		
		\item The following statements are equivalent:
		\begin{enumerate}[{\em(i)}]
			\item $f_{(N)}:X^N\longrightarrow X^N$ is multiply recurrent for every $N \in \NN$;
			
			\item $\com{f}:\Kc(X)\longrightarrow \Kc(X)$ is multiply recurrent;
			
			\item $\fuz{f}:\Fc_{\infty}(X)\longrightarrow \Fc_{\infty}(X)$ is multiply recurrent;
			
			\item $\fuz{f}:\Fc_{0}(X)\longrightarrow \Fc_{0}(X)$ is multiply recurrent.
		\end{enumerate}
		
		\item Given any Furstenberg family $\Ac \subset \Part(\NN_0)$, the following statements are equivalent:
		\begin{enumerate}[{\em(i)}]
			\item $f_{(N)}:X^N\longrightarrow X^N$ is topologically $\Ac$-recurrent for every $N \in \NN$;
			
			\item $\com{f}:\Kc(X)\longrightarrow \Kc(X)$ is topologically $\Ac$-recurrent;
			
			\item $\fuz{f}:\Fc_{\infty}(X)\longrightarrow \Fc_{\infty}(X)$ is topologically $\Ac$-recurrent;
			
			\item $\fuz{f}:\Fc_{0}(X)\longrightarrow \Fc_{0}(X)$ is topologically $\Ac$-recurrent.
		\end{enumerate}
	\end{enumerate}
\end{corollary}
\begin{proof}
	This is a direct consequence of Remark~\ref{Rem:(l,A)-rec} and Theorem~\ref{The:F.rec}.
\end{proof}

As mentioned in Remark~\ref{Rem:K.rec} but this time for Corollary~\ref{Cor:F.rec}, this result is not true for the notion of {\em topological recurrence} (neither for {\em multiple recurrence}) if we do not assume for the continuous map~$f$ the $N$-fold direct product hypothesis for every $N \in \NN$. Indeed, the already mentioned dynamical system $f:X\longrightarrow X$ constructed in \cite[Theorem~3.2]{GriLoPe2025_AMP} for each $N \in \NN$, acts on a metric space $X$ and fulfills that $f_{(N)}$ is topologically recurrent (and even multiply recurrent), but also that $f_{(N+1)}$ is not topologically recurrent (neither multiply recurrent).

\section{The case of complete metric spaces}\label{Sec_5:complete}

In this section we assume that $X$ is a complete metric space. This assumption is usually fulfilled by the phase spaces considered in the literature ($\RR^N$ or $\CC^N$, compact metric spaces, etc.), and it allows us to get stronger results than Corollary~\ref{Cor:F.rec} by using {\em Baire category arguments}. Let us start by recalling which kind of point-recurrence notions we are going to consider.

Given a dynamical system $(X,f)$ we say that a point $x \in X$ is:
\begin{enumerate}[--]
	\item {\em recurrent for $f$}, if for every neighbourhood $U$ of $x$ there exists some (and hence infinitely many) $n \in \NN$ such that $f^n(x) \in U$, i.e.\ if the {\em return set from $x$ to $U$}, which will be denoted by
	\[
	\Nc_f(x,U) := \{ n \in \NN_0 \ ; \ f^n(x) \in U \},
	\]
	is an infinite set. We will denote by $\Rec(f)$ the {\em set of recurrent points for $f$}, and the system $(X,f)$ is called {\em point-recurrent} if the set $\Rec(f)$ is dense in $X$.
	
	\item {\em $\AP$-recurrent for $f$}, if for every neighbourhood $U$ of $x$ the return set $\Nc_f(x,U)$ contains arbitrarily long arithmetic progressions, i.e.\ if for each positive integer $\ell \in \NN$ there exist $n \in \NN$ and $n_0 \in \NN_0$ such that $\{ n_0+jn \ ; \ 0\leq j\leq \ell \} \subset \Nc_f(x,U)$. We will denote by $\AP\Rec(f)$ the {\em set of $\AP$-recurrent points for $f$}, and the system $(X,f)$ is called {\em $\AP$-recurrent} if the set $\AP\Rec(f)$ is dense in $X$.
\end{enumerate}

It is not hard to show that every {\em point-recurrent} dynamical system is {\em topologically recurrent}, and that every {\em $\AP$-recurrent} system is {\em multiply recurrent}. The converse implications are not true in general (see for instance \cite[Example~12.9]{GrPe2011_book}), but hold when~$X$ is a complete metric space by using some {\em Baire category arguments} (see \cite[Proposition~2.1]{CoMaPa2014} and \cite[Lemma~4.8]{KwiLiOpYe2017}). Thus, we have the next:

\begin{corollary}\label{Cor:point.rec}
	Let $f:X\longrightarrow X$ be a continuous map on a complete metric space $(X,d)$. Then:
	\begin{enumerate}[{\em(a)}]
		\item The following statements are equivalent:
		\begin{enumerate}[{\em(i)}]
			\item $f_{(N)}:X^N\longrightarrow X^N$ is point-recurrent for every $N \in \NN$;
			
			\item $\com{f}:\Kc(X)\longrightarrow \Kc(X)$ is point-recurrent;
			
			\item $\fuz{f}:\Fc_{\infty}(X)\longrightarrow \Fc_{\infty}(X)$ is point-recurrent;
			
			\item $\fuz{f}:\Fc_{0}(X)\longrightarrow \Fc_{0}(X)$ is point-recurrent.
		\end{enumerate}
		
		\item The following statements are equivalent:
		\begin{enumerate}[{\em(i)}]
			\item $f_{(N)}:X^N\longrightarrow X^N$ is $\AP$-recurrent for every $N \in \NN$;
			
			\item $\com{f}:\Kc(X)\longrightarrow \Kc(X)$ is $\AP$-recurrent;
			
			\item $\fuz{f}:\Fc_{\infty}(X)\longrightarrow \Fc_{\infty}(X)$ is $\AP$-recurrent;
			
			\item $\fuz{f}:\Fc_{0}(X)\longrightarrow \Fc_{0}(X)$ is $\AP$-recurrent.
		\end{enumerate}
	\end{enumerate}
\end{corollary}
\begin{proof}
	If the metric space $(X,d)$ is complete then we have: that the metric space $(\Kc(X),d_H)$ is complete (see for instance \cite[Exercise~2.15]{IllNad1999}); that $\Fc_{\infty}(X) \equiv (\Fc(X),d_{\infty})$ is complete (see \cite{KlePuRa1986}); and that, even though $\Fc_{0}(X) \equiv (\Fc(X),d_0)$ is not necessarily complete (see \cite[Theorem~3.4]{JoKim2000}), by the result \cite[Theorem~3.8]{JoKim2000} (and its generalization to arbitrary metric spaces \cite[Proposition~4.6]{Huang2022}) there exists an equivalent metric $d_{1}$ (i.e.\ the topology induced by $d_{0}$ and $d_{1}$ coincide) for which $(\Fc(X),d_{1})$ is a complete metric space (see \cite[Theorem~3.9]{JoKim2000}). The result is now a direct consequence of the fact that point-recurrence (resp.\ $\AP$-recurrence) implies topological recurrence (resp.\ multiple recurrence), together with Corollary~\ref{Cor:F.rec} and the completeness of the spaces involved.
\end{proof}

\begin{remark}\label{Rem:point.rec}
	The implication (i) $\Rightarrow$ (ii) of Corollary~\ref{Cor:point.rec} holds when $X$ is a (not necessarily metrizable) topological space, but also (i) $\Rightarrow$ (iii) $\Rightarrow$ (iv) $\Rightarrow$ (ii) are true when $X$ is a (not necessarily complete) metric space. Indeed, the needed arguments are the following:
		\begin{enumerate}[--]
			\item (i) $\Rightarrow$ (ii): Given a basic Vietoris-open set $\Vc(U_1,...,U_N)$, where $U_1,...,U_N$ are $N \in \NN$ arbitrary but fixed non-empty open subsets of $X$, then we can choose a point
			\[
			(x_1,...,x_N) \in U_1\times\cdots\times U_N
			\]
			such that $(x_1,...,x_N) \in \Rec(f_{(N)})$, resp.\ $(x_1,...,x_N) \in \AP\Rec(f_{(N)})$. Considering the compact set $K:=\{x_1,...,x_N\}$ we have that $K \in \Vc(U_1,...,U_N)$ and $K \in \Rec(\com{f})$, resp.\ $K \in \AP\Rec(\com{f})$.
			
			\item (i) $\Rightarrow$ (iii): Given any $u \in \Fc(X)$ and $\eps>0$ we can use Lemma~\ref{Lem:eps.pisos} to obtain positive numbers $0 = \alpha_0 < \alpha_1 < \alpha_2 < ... < \alpha_N = 1$ such that $d_H(u_{\alpha}, u_{\alpha_{i+1}})<\tfrac{\eps}{2}$ for each $\alpha \in \ ]\alpha_i, \alpha_{i+1}]$ with $1\leq i\leq N-1$. Since the topology induced by the Hausdorff metric is equivalent to the Vietoris topology, we can now find non-empty open sets $U_1^i,...,U_{k_i}^i \subset X$ for each $1\leq i\leq N$ such that
			\[
			\Vc(U_1^i,...,U_{k_i}^i) \subset \Bc_H(\alpha_i,\tfrac{\eps}{2}) \quad \text{ for every } 1\leq i\leq N,
			\]
			and by hypothesis we can choose a point
			\[
			(x_1^1,...,x_{k_1}^1,...,x_1^N,...,x_{k_N}^N) \in U_1^1\times\cdots\times U_{k_1}^1\times\cdots\times U_1^N\times\cdots\times U_{k_N}^N
			\]
			with $(x_1^1,...,x_{k_1}^1,...,x_1^N,...,x_{k_N}^N) \in \Rec(f_{(M)})$, resp.\ $(x_1^1,...,x_{k_1}^1,...,x_1^N,...,x_{k_N}^N) \in \AP\Rec(f_{(M)})$, for $M=\sum_{i=1}^N k_i$. Considering now the compact sets $K_i := \{ x_1^i,...,x_{k_i}^i \}$ for each $1\leq i\leq N$, one can check that the fuzzy set $v:=\max_{1\leq i\leq N} (\alpha_i \cdot \chi_{K_i})$ belongs to $\Bc_{\infty}(u,\eps)$ arguing as in Theorem~\ref{The:F.rec}, but also that $v \in \Rec(\fuz{f})$, resp.\ $v \in \AP\Rec(\fuz{f})$, for the metric $d_{\infty}$.
			
			\item (iii) $\Rightarrow$ (iv): Follows from the fact that $d_{0}(u,v) \leq d_{\infty}(u,v)$ for every $u,v \in \Fc(X)$.
			
			\item (iv) $\Rightarrow$ (ii): Given any $K \in \Kc(X)$ and $\eps>0$ we can choose a fuzzy set $u \in \Bc_{0}(\chi_K,\eps)$ such that $u \in \Rec(\fuz{f})$, resp.\ $u \in \AP\Rec(\fuz{f})$, for the metric $d_0$. Considering the set $L := u_1$, and given any non-negative integer $n \in \NN_0$ such that $\fuz{f}^n(u) \in B_{0}(\chi_K,\eps)$, we have that
			\[
			d_H\left(K,\com{f}^n(L)\right) = d_H\left(K,f(u_1)\right) = d_H\left(K,\left[\fuz{f}^n(u)\right]_1\right) \leq d_{\infty}\left(\chi_K,\fuz{f}^n(u)\right) = d_{0}\left(\chi_K,\fuz{f}^n(u)\right) < \eps.
			\]
			We can deduce that $L \in \Bc_H(K,\eps)$ and also that $L \in \Rec(\com{f})$, resp.\ $L \in \AP\Rec(\com{f})$.
		\end{enumerate}
		We do not know if statement (ii) implies any of the other statements of Corollary~\ref{Cor:point.rec}, or if any of the implications not mentioned in this remark holds when $(X,d)$ is not a complete metric space.
\end{remark}

As in Remark~\ref{Rem:K.rec} and at the end of Section~\ref{Sec_4:F(X)}, the conclusion of Corollary~\ref{Cor:point.rec} is not true for the notion of {\em point-recurrence} (neither for {\em $\AP$-recurrence}) if we do not assume for the map $f$ the $N$-fold direct product hypothesis for every $N \in \NN$. Indeed, the already mentioned dynamical systems constructed in \cite[Theorem~3.2]{GriLoPe2025_AMP} are again a counterexample in this case. However, if apart from the completeness of $(X,d)$ we also assume separability, then we can delete the $N$-fold direct product hypothesis by using the notion of quasi-rigidity recently introduced in \cite{GriLoPe2025_AMP}:
\begin{enumerate}[--]
	\item a dynamical system $(X,f)$ is called {\em quasi-rigid} if there exist an increasing sequence of positive integers $(n_k)_{k\in\NN}$ and a dense set $Y \subset X$ such that $f^{n_k}(x) \to x$, as $k\to\infty$, for every $x \in Y$.
\end{enumerate}

The notion of quasi-rigidity has been used in many concrete recurrence problems (see \cite{GriLoPe2025_AMP,Lopez2024_IMRN,LoMe2025_JMAA}), and it is not hard to check that every {\em quasi-rigid} system $(X,f)$ fulfills that $(X^N,f_{(N)})$ is {\em point-recurrent} for every $N \in \NN$. The converse holds when~$X$ is a separable complete metric (also called {\em Polish}) space by the {\em Baire category theorem} (see \cite[Theorem~2.5]{GriLoPe2025_AMP}), and we can obtain the next:

\begin{corollary}\label{Cor:q-r}
	Let $f:X\longrightarrow X$ be a continuous map on a separable complete metric space $(X,d)$. Then, the following statements are equivalent:
	\begin{enumerate}[{\em(i)}]
		\item $f:X\longrightarrow X$ is quasi-rigid;
		
		\item $\com{f}:\Kc(X)\longrightarrow \Kc(X)$ is quasi-rigid;
			
		\item $\fuz{f}:\Fc_{\infty}(X)\longrightarrow \Fc_{\infty}(X)$ is quasi-rigid;
		
		\item $\fuz{f}:\Fc_{0}(X)\longrightarrow \Fc_{0}(X)$ is quasi-rigid.
	\end{enumerate}
\end{corollary}
Note that $\Fc_{\infty}(X) \equiv (\Fc(X),d_{\infty})$ is non-separable as soon as $X$ has more than one point. Hence, and as we are about to show, the only non-trivial equivalent condition in Corollary~\ref{Cor:q-r} is (iii).
\begin{proof}[Proof of Corollary~\ref{Cor:q-r}]
	If $(X,d)$ is separable then we have: that $(\Kc(X),d_H)$ is separable (see for instance \cite[Exercise~1.18]{IllNad1999}); and that $\Fc_{0}(X) \equiv (\Fc(X),d_0)$ is separable (see \cite[Theorem~4.12]{JarSanSan2020a}). Since these spaces are also completely metrizable (see the proof of Corollary~\ref{Cor:point.rec} for details) the equivalences (i) $\Leftrightarrow$ (ii) $\Leftrightarrow$ (iv) follow from the fact that quasi-rigidity implies topological recurrence for every $N$-fold direct product, together with Corollary~\ref{Cor:F.rec}, separability and completeness.
	
	(ii) $\Rightarrow$ (iii): We are assuming that $(\Kc(X),\com{f})$ is quasi-rigid so we can consider $(n_k)_{k\in\NN} \in \NN^{\NN}$ and a dense set $\Yc \subset \Kc(X)$ fulfilling that $\com{f}^{n_k}(K) \to K$, as $k\to\infty$, for every $K \in \Yc$. Now, given $u \in \Fc(X)$ and $\eps>0$ we can use Lemma~\ref{Lem:eps.pisos} to obtain positive numbers $0 = \alpha_0 < \alpha_1 < \alpha_2 < ... < \alpha_N = 1$ such that $d_H(u_{\alpha}, u_{\alpha_{i+1}})<\tfrac{\eps}{2}$ for each $\alpha \in \ ]\alpha_i, \alpha_{i+1}]$ with $1\leq i\leq N-1$, and then we can select compact sets $K_i \in \Bc_H(u_{\alpha_i},\tfrac{\eps}{2}) \cap \Yc$. Hence, the fuzzy set $v:=\max_{1\leq i\leq N} (\alpha_i \cdot \chi_{K_i})$ belongs to $\Bc_{\infty}(u,\eps)$ arguing as in Theorem~\ref{The:F.rec}, and it is not hard to check that $\fuz{f}^{n_k}(v) \to v$, as $k\to\infty$, for the metric $d_{\infty}$.
	
	(iii) $\Rightarrow$ (iv): Trivial since the topology induced by $d_{\infty}$ in $\Fc(X)$ is finer than that induced by $d_{0}$.
\end{proof}

\begin{remark}\label{Rem:q-r}
	The implication (i) $\Rightarrow$ (ii) of Corollary~\ref{Cor:q-r} holds when~$X$ is a (not necessarily metrizable) topological space, but also the \textbf{equivalences} (ii) $\Leftrightarrow$ (iii) $\Leftrightarrow$ (iv) are true when the underlying space~$X$ is a (not necessarily separable, neither complete) metric space. The needed arguments are very similar to those exposed in Remark~\ref{Rem:point.rec}:
	\begin{enumerate}[--]
		\item (i) $\Rightarrow$ (ii): By assumption there exist a sequence $(n_k)_{k\in\NN} \in \NN^{\NN}$ and a dense set $Y \subset X$ such that $f^{n_k}(x)\to x$, as $k\to\infty$, for every $x \in Y$. Then given a basic Vietoris-open set $\Vc(U_1,...,U_N)$, where $U_1,...,U_N$ are $N \in \NN$ non-empty open subsets of $X$, we can choose points $x_j \in U_j \cap Y$ for $1\leq j\leq N$. Thus, the set $K:=\{x_1,...,x_N\}$ belongs to $\Vc(U_1,...,U_N)$ and fulfills that $\com{f}^{n_k}(K)\to K$ as $k\to\infty$.
			
		\item Note that the proofs given in Corollary~\ref{Cor:q-r} for (ii) $\Rightarrow$ (iii) $\Rightarrow$ (iv) are valid on every metric space.
		
		\item (iv) $\Rightarrow$ (ii): By assumption there exist a sequence $(n_k)_{k\in\NN} \in \NN^{\NN}$ and a $d_{0}$-dense set $\Zc \subset \Fc(X)$ such that $\fuz{f}^{n_k}(u) \to u$, for the metric $d_{0}$ and as $k\to\infty$, for every $u \in \Zc$. Then given any $K \in \Kc(X)$ and any $\eps>0$ we can choose a fuzzy set $v \in \Bc_{0}(\chi_K,\eps) \cap \Zc$. Considering the set $L := v_1$, and arguing as in Remark~\ref{Rem:point.rec}, one can check that $L \in \Bc_H(K,\eps)$ and also that $\com{f}^{n_k}(L) \to L$ as $k\to\infty$.
	\end{enumerate}
	We do not know if statements (ii), (iii) and (iv) imply statement (i) of Corollary~\ref{Cor:q-r} when $(X,d)$ is not a separable complete metric space.
\end{remark}

\section{Conclusions and open problems}\label{Sec_6:conclusions}

We have examined the interaction between various individual recurrence-kind properties presented by a discrete dynamical system $(X,f)$ and the respective collective dynamics of its extensions $(\Kc(X),\com{f})$ and $(\Fc(X),\fuz{f})$. In particular, we have proved that the {\em topological $(\ell,\Ac)$-recurrence} for every $(X^N,f_{(N)})$ is equivalent to the fact that $(\Kc(X),\com{f})$ has this same property when $X$ is any topological space, but also equivalent to the {\em topological $(\ell,\Ac)$-recurrence} of $(\Fc(X),\fuz{f})$ with respect to the metrics $d_{\infty}$~and~$d_{0}$ when~$(X,d)$ is a metric space. As a consequence, the same results hold: for the particular notions of {\em topological recurrence}, {\em multiple recurrence} and {\em topological $\Ac$-recurrence} in general; but also for the stronger notions of {\em point-recurrence} and {\em point-$\AP$-recurrence} (also {\em quasi-rigidity}) when $X$ is a complete (and also separable) metric space. Now, we would like to close the paper by mentioning some open problems and possible future lines of research.

First of all, and following Section~\ref{Sec_5:complete}, we would like to highlight in a proper way the left implications mentioned in Remarks~\ref{Rem:point.rec}~and~\ref{Rem:q-r}. Indeed, we have proved there that some of the implications and equivalences stated in Corollaries~\ref{Cor:point.rec} and \ref{Cor:q-r} still hold when the underlying space is not complete. Answering the next open question would extend these results to a much more general class of spaces:

\begin{question}
	Let $X$ be a metric (or even topological) space and assume that $(X,f)$ is a dynamical system for which $\com{f}:\Kc(X)\longrightarrow\Kc(X)$ is point-recurrent ($\AP$-recurrent or quasi-rigid). Does it follow that $f_{(N)}:X^N\longrightarrow X^N$ is point-recurrent ($\AP$-recurrent or quasi-rigid) for every $N \in \NN$?
\end{question}

As a second possible future line of work, and having in mind the works \cite{BerPeRo2017} and \cite{WuXueJi2012} about collective dynamics in the context of Linear Dynamics, one could consider the following situation:
\begin{enumerate}[--]
	\item Given a continuous linear operator $f:X\longrightarrow X$ acting on a topological vector space $X$, let~$\Cc(X)$ be the {\em hyperspace of non-empty compact and convex subsets} of such a space $X$. It is well-known that the hyperextension map $\com{f}:\Kc(X)\longrightarrow \Kc(X)$ leaves the set $\Cc(X)$ invariant by the continuity and linearity of $f$. Moreover, considering $\Cc(X)$ as a topological subspace of $\Kc(X)$ with the restricted Vietoris topology, one can study the dynamics of the restriction
	\[
	\com{f}\vert_{\Cc(X)} := \con{f}:\Cc(X)\longrightarrow\Cc(X),
	\]
	where $\con{f}(K) := \com{f}(K) = f(K) = \{ fx \ ; \ x \in K \}$ for every set $K \in \Cc(X)$.
\end{enumerate}
The following natural open question arises (see \cite{BerPeRo2017} for more details about the setting considered):

\begin{question}
	Let $\ell \in \NN$ be a positive integer, $\Ac \subset \Part(\NN_0)$ a Furstenberg family, and $f:X\longrightarrow X$ a continuous linear operator acting on a complete locally convex topological vector space $X$. Is the system $\com{f}:\Kc(X)\longrightarrow \Kc(X)$ topologically $(\ell,\Ac)$-recurrent if and only if so is $\con{f}:\Cc(X)\longrightarrow \Cc(X)$?
\end{question}

As a third possible future line of research we must mention the role of invariant measures. Indeed, in 1975, Bauer and Sigmund already studied the existence of invariant measures for the extended system~$(\Kc(X),\com{f})$, showing that it admits an $\com{f}$-invariant measure if and only if the system $(X,f)$ admits an $f$-invariant measure (see \cite[Proposition~5]{BaSig1975}). As far as we know, this topic has not been considered in the context of the fuzzy extension $(\Fc(X),\fuz{f})$, and since the existence of invariant measure has been recently related to some strong point-recurrence properties (see \cite{GriLo2023} and \cite{Lopez2024_RinM}), we consider interesting to mention it here: in \cite{GriLo2023} it was proved that the existence of an invariant measure with full support, for a dynamical system $(X,f)$ where $X$ is a second-countable space, implies that the system~$(X,f)$ is {\em frequently recurrent} (a point-recurrence property stronger than $\AP$-recurrence). Hence, the following question is then natural (see \cite{GriLo2023,Lopez2024_RinM} for more details about the role of invariant measures in recurrence):

\begin{question}
	Let $(X,f)$ be a dynamical system acting on a metric space $(X,d)$. Under which conditions do the extended systems $(\Fc_{\infty}(X),\fuz{f})$ and/or $(\Fc_{0}(X),\fuz{f})$ admit an $\fuz{f}$-invariant measure?
\end{question}

Finally, we provide a detailed explanation of the topologies chosen for the hyperspaces used in this paper $\Kc(X)$ and $\Fc(X)$, since changing the topologies to check if the results obtained still hold could potentially be another future line of research:
\begin{enumerate}
	\item[--] For $\Kc(X)$, the hyperspace of non-empty compact subsets of a topological space $X$, we adopted the Vietoris topology, in line with previous works \cite{BaSig1975,BerPeRo2017,Peris2005,WuXueJi2012}. This choice is standard in the literature for studying the dynamical properties of $(\Kc(X), \com{f})$.
	
	\item[--] For $\Fc(X)$, the space of normal fuzzy sets of $X$ with compact support, we employed the metrics $d_{\infty}$ and $d_{0}$, as seen in references such as \cite{CaKup2017,ChaRo2008,JarSan2021_FS,JarSanSan2020a,JarSanSan2020b,JoKim2000,KimChenJu2017,Kup2011,LiWaZha2006,MarPeRo2021,RiKimJu2023,WuDingLuWang2017}. These metrics slightly rely on the assumption that the fuzzy sets have compact support (though some works further restrict the support to be \textbf{compact and convex}, particularly in linear cases). While some of these references consider additional metrics beyond $d_{\infty}$ and $d_{0}$, they all assume that $X$ is a metric space, and therefore so is $\Fc(X)$. Notably, even non-dynamical studies on fuzzy sets, such as \cite{Huang2022}, exclusively consider metrics and do not explore non-metrizable topologies for different spaces of fuzzy sets.
\end{enumerate}

These observations arise from a significant point raised by one of the anonymous reviewer of this paper: in \cite{WangWei2012}, the authors did not rely on any metric to study the dynamics of $(\Fc(X), \fuz{f})$. Instead, the hyperspaces $\Kc(X)$ and $\Fc(X)$ were equipped with the so-called {\em hit-or-miss topology}, which was used to analyze dynamical properties such as Li-Yorke chaos, $\omega$-chaos, and distributional chaos for the systems $(\Kc(X), \com{f})$ and $(\Fc(X), \fuz{f})$. Importantly, the authors of \cite{WangWei2012} acknowledged that this topology has not been widely applied in fuzzy dynamics, and their work appears largely independent of the foundational studies that we have referenced for this paper.

In light of the previous fact, we propose a potential and final avenue for future research: exploring the various fuzzy dynamical properties of the system $(\Fc(X), \fuz{f})$ and existing in the literature under other metrics from those considered in the references cited in this paper, but also with respect to the hit-or-miss topology (see again \cite{WangWei2012}). We find very remarkable that this last topology does not require restricting $\Fc(X)$ to fuzzy sets with compact support, fact that could potentially leading to some brand new or completely different results.

\section*{Funding}

The second and third authors were supported by MCIN/AEI/10.13039/501100011033/FEDER, UE, Project PID2022-139449NB-I00. The second author was also partially supported by the Spanish Ministerio de Ciencia, Innovaci\'on y Universidades, grant FPU2019/04094. The third author was also partially supported by Generalitat Valenciana, Project PROMETEU/2021/070.

\section*{Acknowledgments}

The authors would like to thank the anonymous reviewers, whose careful comments have significantly improved not only the presentation of this paper but also the selection of references included.

{\footnotesize
$\ $\\

\textsc{Illych Alvarez}: Escuela Superior Polit\'ecnica del Litoral, Facultad de Ciencias Naturales y Matem\'aticas, Km.\ 30.5 V\'ia Perimetral, Guayaquil, Ecuador. e-mail: ialvarez@espol.edu.ec

\textsc{Antoni L\'opez-Mart\'inez}: Universitat Polit\`ecnica de Val\`encia, Institut Universitari de Matem\`atica Pura i Aplicada, Edifici 8E, 4a planta, 46022 Val\`encia, Spain. e-mail: alopezmartinez@mat.upv.es

\textsc{Alfred Peris}: Universitat Polit\`ecnica de Val\`encia, Institut Universitari de Matem\`atica Pura i Aplicada, Edifici 8E, 4a planta, 46022 Val\`encia, Spain. e-mail: aperis@mat.upv.es
}

\end{document}